\documentclass[12pt]{article}

\usepackage[usenames]{color}
\usepackage{amssymb}
\usepackage{amsmath}
\usepackage{amsthm}
\usepackage{amsfonts}
\usepackage{amscd}
\usepackage{graphicx}

\usepackage[colorlinks=true,
linkcolor=webgreen,
filecolor=webbrown,
citecolor=webgreen]{hyperref}

\definecolor{webgreen}{rgb}{0,.5,0}
\definecolor{webbrown}{rgb}{.6,0,0}

\usepackage{color}
\usepackage{fullpage}
\usepackage{float}

\usepackage{graphicx}
\usepackage{graphics}
\usepackage{latexsym}
\usepackage{epsf}
\usepackage{breakurl}

\newcommand{\seqnum}[1]{\href{https://oeis.org/#1}{\rm \underline{#1}}}

\usepackage{tikz}
\usepackage{tkz-tab}
\usetikzlibrary{matrix}
\usetikzlibrary{calc}
\usetikzlibrary{fit}

\begin{document}

\theoremstyle{plain}
\newtheorem{theorem}{Theorem}
\newtheorem{corollary}[theorem]{Corollary}
\newtheorem{lemma}[theorem]{Lemma}
\newtheorem{proposition}[theorem]{Proposition}

\theoremstyle{definition}
\newtheorem{definition}[theorem]{Definition}
\newtheorem{example}[theorem]{Example}
\newtheorem{conjecture}[theorem]{Conjecture}

\theoremstyle{remark}
\newtheorem{remark}[theorem]{Remark}

\begin{center}
\vskip 1cm{\LARGE\bf A New Identity Linking Bernoulli Numbers, Stirling Numbers of the First Kind, and Bessel Numbers of the First Kind} \vskip 1cm \large
Abdelhay Benmoussa\\
Primary School Teacher\\
S S Ansis\\
Agadir Ida Outanane 80000\\
Morocco\\
\href{mailto:abdelhay.benmoussa@taalim.ma}{\tt abdelhay.benmoussa@taalim.ma}
\end{center}
\vskip .2 in

\begin{flushright}
\textit{To my family.}
\end{flushright}
\begin{abstract}
We establish a new identity linking Bernoulli, Stirling (first kind), and Bessel (first kind) numbers:
\[
\sum_{k=0}^{n} 2^{\,n-k}\,s(n,k)\,B_k \;=\; \sum_{k=0}^{n} b(n,k)\,\frac{(-1)^k\,k!}{k+1}.
\]
This parallels the classical Stirling--Bernoulli relation
\[
B_n = \sum_{k=0}^{n} S(n,k)\,\frac{(-1)^k\,k!}{k+1},
\]
replacing $S(n,k)$ with $s(n,k)$ and $b(n,k)$, and thus revealing a new structural connection among these families of numbers.
\end{abstract}

\section{Introduction}
A classical identity expresses Bernoulli numbers via Stirling numbers of the second kind:
\begin{equation}\label{eq:classical}
B_n = \sum_{k=0}^{n} S(n,k)\,\frac{(-1)^k\,k!}{k+1}.
\end{equation}
In this paper we present a close analogue involving Stirling numbers of the first kind and Bessel numbers of the first kind:
\begin{equation}\label{eq:ours}
\sum_{k=0}^{n} 2^{\,n-k}\,s(n,k)\,B_k
= \sum_{k=0}^{n} b(n,k)\,\frac{(-1)^k\,k!}{k+1}.
\end{equation}
Both identities feature the same factor $\frac{(-1)^k k!}{k+1}$ on the right, while the left involves either a single Bernoulli number $B_n$ or a Stirling-first-kind convolution of Bernoulli numbers. 
This striking formal parallel hints at a deeper ``first-kind/second-kind'' duality between the two identities.

\section{Notation and preliminaries}

\paragraph{Bernoulli numbers.}
The Bernoulli numbers $B_n$ are defined via the generating function
\[
\frac{t}{e^t-1}=\sum_{n=0}^{\infty} B_n \frac{t^n}{n!},\qquad \text{for all } |t|<2\pi.
\]

\paragraph{Stirling numbers of the first kind.}
We denote by $s(n,k)$ the Stirling numbers of the first kind, characterized by the falling factorial expansion \cite{Comtet}
\begin{equation}\label{eq:falling}
(x)_n := x(x-1)\cdots(x-n+1) \;=\; \sum_{k=0}^{n} s(n,k)\,x^k,
\qquad s(n,1)=(-1)^{n-1}(n-1)!.
\end{equation}

\paragraph{Bessel numbers of the first kind.}
We adopt the explicit definition \cite{han}
\begin{equation}\label{eq:b-explicit}
b(n,k) \;=\;
\begin{cases}
(-1)^{n-k}\,\dfrac{(2n-k-1)!}{2^{\,n-k}\,(k-1)!\,(n-k)!}, & \text{if } 1\le k\le n,\\[2mm]
0, & \text{if } 0\le n<k.
\end{cases}
\end{equation}
Equivalently, the $b(n,k)$ appear as coefficients in the operational identity 
\begin{equation}\label{eq:oper}
(x^{-1}D)^n \;=\; \sum_{k=0}^{n} b(n,k)\,x^{\,k-2n}\,D^k,
\qquad D=\frac{d}{dx}.
\end{equation}
The sequence $b(n,k)$ is cataloged as OEIS \seqnum{A122850}.

\section{Main result}

\begin{lemma}\label{lem:quad-prod-to-power} For every $n\ge0$, \begin{equation}\label{eq:quad-to-power} \prod_{j=0}^{n-1}\bigl(t-2j\bigr) \;=\; \sum_{k=0}^{n} 2^{\,n-k}\,s(n,k)\,t^k. \end{equation} \end{lemma} \begin{proof} Write \[ \prod_{j=0}^{n-1} (t-2j) = 2^n \prod_{j=0}^{n-1}\!\left(\frac{t}{2}-j\right) = 2^n \left(\frac{t}{2}\right)_{n} = \sum_{k=0}^{n} 2^{n-k}\,s(n,k)\,t^k, \] which proves the lemma. \end{proof}

\begin{theorem}\label{thm:main}
For every integer $n\ge0$, the following identity holds:

\begin{equation}
\sum_{k=0}^{n} 2^{\,n-k}\,s(n,k)\,B_k=\sum_{k=0}^{n} b(n,k)\,\frac{(-1)^k\,k!}{k+1}, 
\label{eq:main1}
\end{equation}
\end{theorem}

\begin{proof}
Applying \eqref{eq:oper} to the monomial $x^t$ gives
\[
t(t-2)(t-4)\cdots\bigl(t-2(n-1)\bigr)
=\sum_{k=0}^{n} b(n,k)\,(t)_k,
\]
and by Lemma~\ref{lem:quad-prod-to-power} we also have
\[
\prod_{j=0}^{n-1}(t-2j)
=\sum_{k=0}^{n} 2^{\,n-k}\,s(n,k)\,t^k.
\]
Therefore,
\begin{equation}\label{eq:two-exp}
\sum_{k=0}^{n} 2^{\,n-k}\,s(n,k)\,t^k
=\sum_{k=0}^{n} b(n,k)\,(t)_k.
\end{equation}

Summing \eqref{eq:two-exp} over $t=0,1,\dots,m-1$ gives
\[
\sum_{k=0}^{n} 2^{\,n-k}\,s(n,k)\,\sum_{t=0}^{m-1} t^k
=\sum_{k=0}^{n} b(n,k)\,\sum_{t=0}^{m-1}(t)_k.
\]

On the left-hand side, Faulhaber’s formula yields
\[
\sum_{t=0}^{m-1} t^k
=\frac{1}{k+1}\sum_{j=0}^{k} \binom{k+1}{j} B_j\,m^{\,k+1-j},
\]
so the coefficient of $m^1$ is exactly $B_k$.  
On the right-hand side, the telescoping identity $\Delta (t)_{k+1}=(k+1)(t)_k$ gives
\[
\sum_{t=0}^{m-1}(t)_k=\frac{(m)_{k+1}}{k+1}.
\]
Expanding $(m)_{k+1}$ into powers of $m$ using Stirling numbers of the first kind shows that the coefficient of $m^1$ is $(-1)^k k!$.  

Equating the coefficients of $m^1$ on both sides proves
\[
\sum_{k=0}^{n} 2^{\,n-k}\,s(n,k)\,B_k
=\sum_{k=0}^{n} b(n,k)\,\frac{(-1)^k k!}{k+1},
\]
which is \eqref{eq:main1}.

\end{proof}

\begin{corollary}\label{cor:main}
For integers $n\ge0$ and $0\le i\le n$,

\begin{equation}\label{eq:cor}
\sum_{k=i}^{n} 2^{\,n-k}\,s(n,k)\,S(k,i)
\;=\;
b(n,i),
\end{equation}
Here $S(k,i)$ denotes a Stirling number of the second kind.

\end{corollary}

\begin{proof}
Recall the well-known expansion of Bernoulli numbers in terms of Stirling numbers of the second kind:
\[
B_k \;=\; \sum_{i=0}^{k} S(k,i)\,\frac{(-1)^i\,i!}{i+1}.
\]
Substitute this into the left-hand side of Identity~\eqref{eq:main1}:
\[
\sum_{k=0}^{n} 2^{\,n-k}\,s(n,k)\,B_k
= \sum_{k=0}^{n} 2^{\,n-k}\,s(n,k)
   \sum_{i=0}^{k} S(k,i)\,\frac{(-1)^i\,i!}{i+1}.
\]
Interchanging summations yields
\[
\sum_{i=0}^{n} \frac{(-1)^i\,i!}{i+1}
   \sum_{k=i}^{n} 2^{\,n-k}\,s(n,k)\,S(k,i).
\]
By Theorem~\ref{thm:main}, this equals
\[
\sum_{i=0}^{n} b(n,i)\,\frac{(-1)^i\,i!}{i+1}.
\]
Since the factors $\frac{(-1)^i i!}{i+1}$ are independent, identity \eqref{eq:cor} follows \cite{Stenlund}.

This completes the proof.
\end{proof}

\section*{Acknowledgment}
The author originally discussed this identity and its proof on MathOverflow.

\bigskip
\hrule
\bigskip
\noindent 2020 \textit{Mathematics Subject Classification}: Primary 11B68, 05A18, 11B73, 33C10.

\noindent\emph{Keywords:} Bernoulli number, Stirling number, Bessel number of first kind, operational identity.

\bigskip
\hrule
\bigskip
\noindent (Concerned with sequences
\seqnum{A008275}, \seqnum{A008277}, \seqnum{A122850}.)
\bigskip
\hrule
\bigskip
\vspace*{+.1in}
\noindent
Received ; revised versions received ; . 
Published in \textit{Journal of Integer Sequences}, .
\bigskip
\hrule
\bigskip
\noindent
Return to \href{https://cs.uwaterloo.ca/journals/JIS/}{Journal of Integer Sequences home page}.
\vskip .1in

\begin{thebibliography}{99}

\bibitem{Comtet} L. Comtet, {\it Advanced Combinatorics}, Reidel, 1974.

\bibitem{han} H. Han and S. Seo, Combinatorial proofs of inverse relations and log-concavity for Bessel numbers, {\it European J. Combin.} {\bf 29} (2008), 1544--1554.


\bibitem{Stenlund} D. Stenlund, On the Connection Between Stirling Numbers and Bessel Numbers, {\it Electron. J. Combin.} {\bf 29}(1):P1.40, 2022.

\end{thebibliography}
\end{document}